\documentclass[12pt]{amsart}

\usepackage[T2A]{fontenc}
\usepackage[utf8]{inputenc}
\usepackage[english]{babel}

\usepackage{amsmath}
\usepackage{amssymb}
\usepackage{amscd}
\usepackage{amsfonts,amsmath,amsthm}

\usepackage{yfonts}
\usepackage[matrix,arrow,curve]{xy}
\usepackage{xfrac}
\usepackage{faktor}
\usepackage{color}
\usepackage{turnstile}

\usepackage{enumitem}

\DeclareMathOperator{\V}{Vert}

\DeclareMathOperator{\Cen}{PR}

\DeclareMathOperator{\sign}{sign}

\DeclareMathOperator{\Vol}{Vol}

\newtheorem{thm}{Theorem}

\newtheorem{lemma}{Lemma}
\newtheorem*{main lemma}{The main lemma}

\newtheorem{dfn}{Definition}

\title{Cyclopermutohedron: geometry and topology }

\author{Ilia Nekrasov, Gaiane Panina, Alena Zhukova}
\address{Ilia Nekrasov: St. Petersburg State university, geometr.nekrasov@yandex.ru;\newline Gaiane Panina: SPIIRAN, St. Petersburg State university, gaiane-panina@rambler.ru; \newline Alena Zhukova: St. Petersburg State university,  a.zhukova@spbu.ru}
\keywords{Permutohedron,  virtual polytope, discrete Morse theory, Abel polynomial. \ \   MSC 51M20 }
\begin{document}

\begin{abstract}

The face poset of the \textit{permutohedron} realizes the combinatorics of linearly ordered partitions of the set $[n]=\{1,...,n\}$. Similarly, the \textit{cyclopermutohedron} is a virtual polytope that realizes the combinatorics of cyclically ordered partitions of the set $[n+1]$. The cyclopermutohedron  was introduced by the third author by motivations coming from \textit{configuration spaces of polygonal linkages}.

In the paper we prove two facts: (1)~the volume of the cyclopermutohedron equals zero, and (2) the homology groups $H_k$ for $k=0,...,n-2$ of the face poset of the cyclopermutohedron are non-zero free abelian groups. We also present  a short formula for their ranks.

\end{abstract}

\maketitle

\section{Introduction}\label{SectIntro}

The \textit{standard permutohedron} $\Pi_n$ (see
\cite{z}) is  defined as the convex hull of all points in $\mathbb{R}^n$ that
are obtained by permuting the coordinates of the point
$(1,2,...,n)$. It has the following properties:
\begin{enumerate}

    \item \begin{enumerate}
            \item The $k$-faces of  $\Pi_n$ are labeled by ordered partitions of the set $[n]:=\{1,2,...,n\}$ into $(n-k)$ non-empty parts.

            \item A face $F$ of $\Pi_n$ is contained in a face $F'$ iff the label of $F$ refines the label of $F'$.
		
		   \end{enumerate}
\end{enumerate}

Here and in the sequel, by a \textit{refinement} we mean an \textit{order preserving refinement}. For instance, the label $\left(\{1,3\},\{5,6\},\{4\},\{2\}\right)$ refines the label $(\{1,3\},\{5,6\},\{2,4\}) $ but does not refine $(\{1,3\},\{2,4\},\{5,6\})$.

\begin{enumerate}[resume]
\item $\Pi_n$ is an $(n-1)$-dimensional polytope.

\item $\Pi_n$ is a \textit{zonotope}, that is,  Minkowski sum of line segments $q_{ij}$, whose defining vectors are $\{e_i-e_j\}_{1\leq i < j \leq n}$, where $\{e_i\}_{1\leq i \leq n}$ are the standard orthonormal basis vectors.

\end{enumerate}

By  analogy, we replace the linear order by cyclic order and build up the following  regular\footnote{ To define a regular cell complex, it suffices to list all the closed cells of the complex together with the incidence relations.} cell complex ${CP}_{n+1}$ \cite{pan3}, see Fig.  \ref{FigExLabel}.
		\begin{enumerate}
   		\item Assume that $n>2$. For  $k=0,...,n-2$, the $k$-dimensional cells ($k$-cells, for short) of the complex ${CP}_{n+1}$ are labeled by (all possible) cyclically ordered partitions of the set $[n+1]=\{1,...,n+1\}$ into $(n-k+1)$ non-empty parts.
		\item A (closed) cell $F$ contains a cell $F'$ whenever the label of $F'$ refines the label of $F$. Here we again mean order preserving refinement.
 		\end{enumerate}

The \textit{cyclopermutohedron} $\mathcal{CP}_{n+1}$ is a virtual polytope whose face poset is combinatorially isomorphic to complex ${CP}_{n+1}$. More details will be given in  Section \ref{SSecVir}; for a complete presentation see \cite{pan3}.

\begin{figure}[h]
\centering
\includegraphics[width=8 cm]{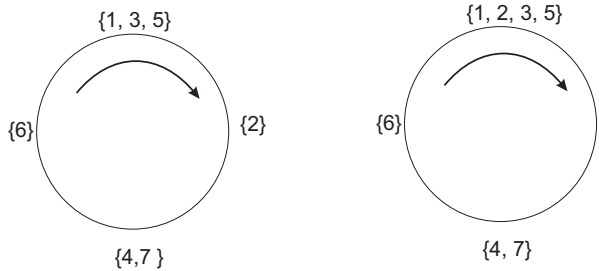}

\caption{These are labels of two cells whose dimensions are $3$ and $4$. The first cell is a face of the second one.
}
\label{FigExLabel}
\end{figure}

In the paper we study geometry and topology of the cyclopermutohedron. Before we formulate the main result some explanation is needed.
The cyclopermutohedron is a \textit{virtual polytope}, that is, the Minkowski difference of two convex polytopes.
 A
detailed discussion on virtual polytopes can be found in the survey
\cite{vir}. One of the messages of the survey is that  virtual polytopes inherit almost all the properties and structures of convex polytopes: the volume (together with its polynomiality property),
normal fan,  face poset, etc.
However, virtual polytopes do not inherit the convexity property and therefore  may appear as counter-intuitive:
(1) The volume of a virtual polytope, although well-defined, can be negative, see \cite{mar, vir}. The volume also can turn to zero, even if the virtual polytope does not degenerate.
(2) The face poset of a virtual polytope is also well-defined. However, it is not necessarily isomorphic to a combinatorial sphere. So one can expect non-zero homologies in all dimensions.

\medskip

 The\textbf{ main results} of the paper are:
\begin{thm}\label{Volume}
The volume of the cyclopermutohedron $\mathcal{CP}_{n+1}$ equals zero.
\end{thm}

\begin{thm}\label{Betti}
The homology groups of the face poset of cyclopermutohedron $\mathcal{CP}_{n+1}$ are free abelian groups. Their ranks are:
$$rk\,(\,H_{k}( CP_{n+1},\mathbb{Z})\,) =\left\{
                                      \begin{array}{lllll}
                                      \dbinom{n}{k}, & \hbox{ if } 0 \leq k < n-2; \hspace{10mm}\\
 &\\
2^n+\frac{n^2-3n-2}{2}, & \hbox{ if } k = n-2; \\
 &\\

                                       0, & \hbox{otherwise}.
                                       \end{array}
                                       \right.$$
\end{thm}

\bigskip

 Let us understand the meaning of the theorems for the toy  example $n=3$, that is, for $\mathcal{CP}_{4}$. The complex ${CP}_{4}$ (and therefore, the face poset of the cyclopermutohedron) is the graph with six vertices and twelve edges, see  Fig. \ref{FigCycloper4}, left.
Its Betti numbers are $1$ and $7$.

The cyclopermutohedron $\mathcal{CP}_{4}$ (computed in \cite{pan3}) can be represented by a closed polygon, whereas its area (that is, two-dimensional volume) equals the integral of the winding number against the Lebesgue measure (see \cite{vir}). In other words, in this case the volume equals ''sum of areas of six small triangles
minus the area of the hexagon'' in Fig. \ref{FigCycloper4} (middle),   which is zero.

\begin{figure}[h]
\centering
\includegraphics[width=12 cm]{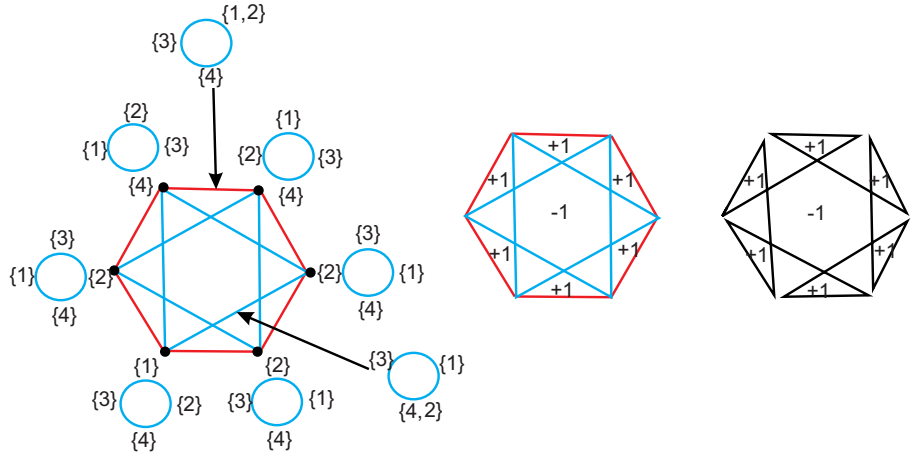}
\caption{Left: face poset of $\mathcal{CP}_{4}$. We indicate the labels of all the vertices and the labels of two edges. Middle: the cyclopermutohedron $\mathcal{CP}_{4}$ represented by a closed polygon, the limit of the polygon depicted on the right. }\label{FigCycloper4}
\end{figure}

\medskip

We use the following \textbf{methods}:
(1) Proof  of Theorem 1 is based on the polynomiality of the volume combined with the theory of Abel polynomials. The proof resembles the volume computation of  the standard permutohedron, which reduces to
count of spanning trees.  
(2) Theorem 2 is proven via discrete Morse theory  by R. Forman \cite{for}. In a sense, it is a simplification of the approach of \cite{zhu} where a perfect discrete Morse function on the moduli space of a polygonal linkage was constructed.

\medskip

\textbf{Acknowledgements.} The present research  is  supported by RFBR, research project No. 15-01-02021. The first author is also supported by   JSC ''Gazprom Neft''. The third author is also supported by «Young Russian Mathematics» foundation.

\newpage
\section{Theoretical backgrounds and toolboxes}\label{SecTheorBackgr}
\subsection{Virtual polytopes}\label{SSecVir}\cite{vir}
Virtual polytopes appeared in the literature as useful geometrization of Minkowski differences of convex polytopes; below we give just a brief sketch.

 \textit{A convex polytope }is the convex hull of a finite, non-empty point set in the Euclidean space $\mathbb{R}^n$.
Minkowski addition turnes the set of all convex polytopes $\mathcal{P}^+$ into a commutative semigroup whose unit element is the  single-point polytope $E= \{ 0 \}$.

\begin{dfn}
The group $\mathcal{P}$ of {\em virtual polytopes} is the Grothendieck group associated to the semigroup $\mathcal{P}^+$.
The elements of $\mathcal{P}$ are called {\em virtual polytopes}.
\end{dfn}
More instructively,  $\mathcal{P}$ can be explained as follows.

\begin{enumerate}
  \item A virtual polytope is a formal difference $K- L$ of convex polytopes.
  \item Two such expressions $K_1 - L_1$ and $K_2 - L_2$ are identified, whenever $K_1 + L_2 = K_2 + L_1$ as convex polytopes.
  \item  The group operation is defined by
  $$(K_1 - L_1) + (K_2 - L_2) := (K_1 + K_2)- (L_1 + L_2).$$
\end{enumerate}

\newpage
\subsection{Cyclopermutohedron}\cite{pan3}

Assuming that $\{e_i\}_{i=1}^n$ are standard basis vectors in $\mathbb{R}^n$, define the points
$$\begin{array}{ccccccccc}
R_i=\sum_{i=1}^n (e_j-e_i)=(-1, & ... & -1, & n-1, & -1, & ... & -1, & -1, &-1, )\in \mathbb{R}^{n},\\
     &  & & \ i &  &  &  &  &
\end{array}
$$
and  the following two families of line segments:
$q_{ij}=\left[e_i,e_j\right], \ \ \  i<j$ and $  r_i=\left[0,R_{i} \right].$
We also need the point $e=\left(1,1,...,1\right)\in \mathbb{R}^{n}$.
The \textit{cyclopermutohedron} is a virtual polytope defined as the Minkowski sum:
$$\mathcal{CP}_{n+1}:= \bigoplus_{i< j} q_{ij} + e - \bigoplus_{i=1}^n  r_i.$$

Throughout the paper the sign ''$ \bigoplus$'' denotes the Minkowski sum, whereas  the sign ''$\sum$'' is reserved for the sum of numbers.

\medskip

The cyclopermutohedron  lies in the hyperplane
$x_1+...+x_n=\frac{n(n+1)}{2},$
so its  dimension is $(n-1)$. Therefore, by its volume we mean the $(n-1)$- volume.

The face poset of $\mathcal{CP}_{n+1}$ is isomorphic to the complex $CP_{n+1}$ (defined in the Introduction).

\subsection{Abel polynomials and rooted forests}\cite{Sagan}\label{SecAbel}

A \textit{rooted forest} is a graph equals to a disjoint union of trees, where each of the trees has one marked vertex.

The \textit{Abel polynomials} form a sequence of polynomials, where the $n$-th term is defined by
$$A_{n,a}(x)=x(x-an)^{n-1}.$$

A special case of the Abel polynomials with $a=-1$ counts rooted labeled forests. Namely, if  $A_{n}(x) := A_{n,-1}(x) = x(x+n)^{n-1}$ is the $n$-th {Abel polynomial}, then
$$A_{n}(x)=\sum_{k = 0}^{n} t_{n,k}\cdot x^{k}  ,$$
where $t_{n,k}$ is the number of forests on $n$ labeled vertices consisting of $k$ rooted trees.

\subsection{Discrete Morse theory on a regular cell complex}\cite{for}\label{SSecDMF}

Assume we have a regular cell complex $X$. By $\alpha^p, \ \beta^p$ we denote its $p$-dimensional cells, or \textit{$p$-cells}, for short.

\textit{A discrete vector field} on $X$ is a set of pairs
$\big(\alpha^p,\beta^{p+1}\big)$
such that:
\begin{enumerate}
    \item  each cell of the complex participates in at most one pair,
    \item  in each pair, the cell $\alpha^p$ is a facet of the cell $\beta^{p+1}$.
\end{enumerate}

Given a discrete vector field, a \textit{gradient path} is a sequence of cells
$$ \beta_0^{p+1},\ \alpha_1^p,\ \beta_1^{p+1}, \ \alpha_2^p,\ \beta_2^{p+1} ,\dots,\ \alpha_m^p,\ \beta_m^{p+1},\ \alpha_{m+1}^p,$$
which satisfies the conditions:
\begin{enumerate}
    \item each $\big(\alpha_i^p,\ \beta_i^{p+1}\big)$ is a pair;
    \item  $\alpha_i^p$ is a facet of  $\beta_{i-1}^{p+1}$;
    \item $\alpha_i \neq \alpha_{i+1}$ for any $i$.
\end{enumerate}

A path is \textit{closed } if $\alpha_{m+1}^p$ is paired with $\beta_0^{p+1}$.
\textit{A discrete Morse function } is a discrete vector field without closed paths.

Assuming that a discrete Morse function is fixed, the \textit{critical cells} are those cells of the complex that are not paired. Morse inequality says that we cannot avoid them completely. However, our goal is to minimize their number.

Discrete Morse function theory allows to compute homology groups. Fix an orientation for each of the cells of complex and introduce the free abelian groups  $\mathcal{M}_k(X,\mathbb{Z})$ whose generators bijectively correspond to critical cells of index $k$.
These groups are incorporated in the chain complex called\textit{ Morse complex }associated with $X$
$$\cdots \rightarrow \mathcal{M}_k(X,\mathbb{Z}) \rightarrow \mathcal{M}_{k-1}(X,\mathbb{Z})\rightarrow \cdots ,$$
where the  boundary operators ${\partial_k}$ are defined by
$$\partial_k(\beta^k) = \sum_{\alpha^{k-1}} [\beta:\alpha] \cdot \alpha^{k-1},$$
where $\alpha$ ranges over all $(k-1)$-dimensional cells, and $[\beta:\alpha]$ is the number of gradient paths from  $\beta^k$ to $\alpha^{k-1}$. Each gradient path is counted with  a sign $\pm1$, depending on whether the orientation of $\beta^k$ induces the chosen orientation on $\alpha^{k-1}$, or the opposite orientation. With this boundary operators the above complex computes the homology groups of $X$:
$$ H_k(X,\mathbb{Z})=Ker(\partial_k)/Im(\partial_{k-1}).$$

\section{Volume of cyclopermutohedron equals zero}\label{SecVolIsZero}

As we have already mentioned, the notion of volume extends nicely from convex polytopes to virtual polytopes. We explain here the meaning of the \textit{volume of a virtual zonotope}.

Assume we have a convex zonotope $Z\subset \mathbb{R}^n$, that is, the Minkowski sum of some linear segments $\{s_i\}_{i=1}^m$ :
$$Z=\bigoplus_{i=1}^m \  s_i.$$

For each subset $I \subset [m]$ such that $|I| = n$, denote by $Z_I$ the \textit{elementary parallelepiped}, or the \textit{brick} spanned by $n$ segments $\{s_i\}_{i\in I}$, provided that the defining vectors of the segments are linearly independent. In other words, the brick equals the Minkowski sum
$$Z_I=\bigoplus_Is_i.$$

It is known that $Z$ can be partitioned into the union of all such $Z_I$, which implies immediately
$$\Vol(Z) = \sum_{I\subset [m],\,|I|=n} \Vol(Z_I) = \sum_{I\subset [m],\,|I|=n} |Det(S_I)|,$$
where $S_I$ is the matrix composed of defining vectors of the segments from $I$.
Now take positive $\lambda_1,...,\lambda_m$ and sum up the dilated segments $\lambda_is_i$. Clearly, we have

$$\Vol\Big(\bigoplus_{i=1}^m\ \lambda_i s_i\Big)=\sum_{I\subset [m], |I|=n}\prod_{i\in I}\lambda_i \cdot|Det(S_I)|.$$

For fixed $s_i$, we get a polynomial in $\lambda_i$, which counts not only the volume of convex zonotope (which originates from positive $\lambda_i$), but also the volume of a virtual zonotope, which originates from any real $\lambda_i$, including negative ones, see \cite{vir}.
So, one can use the above formula as the definition of the volume of a virtual zonotope.

\begin{lemma}\label{lemmaVol}Let $E=E_n$ be the set of edges of the complete graph $K_n$. The $(n-1)$-volume of the cyclopermutohedron can be computed by the formula:
$$\Vol(\mathcal{CP}_{n+1})=  Vol\Big( \bigoplus_{i< j} q_{ij} - \bigoplus_{i=1}^n
 r_i\Big)=$$
$$=\frac{1}{\sqrt{n}}\sum_{|I|+|M|=n-1} (-1)^{|M|}|Det(q_{ij},r_k,e)|_{(ij)\in I, \ k \in M}.$$
Here $I$ ranges over subsets of $E$, whereas $M$ ranges over subsets of $[n]$. The matrix under determinant is composed of defining vectors of the segments $q_{ij}$ and $r_k$, and also of the vector $e=(1,1,...,1,1)$.
\end{lemma}

\begin{proof} The dimension of the cyclopermutohedron  equals $n - 1$. That is, we deal with $(n-1)$-volume, which reduces to the $n$-volume by adding $e=(1,1,...,1,1,)$ and dividing by $|e|=\sqrt{n}$.
\end{proof}

Now we are ready to prove Theorem 1. Keeping in mind Lemma \ref{lemmaVol}, let's first fix $I$ and $M$ with $|I|+|M|=n-1$, and compute one single summand $|Det(q_{ij},  r_{k},  e)|_{(ij)\in I, \ k \in M}$.

If $M=\emptyset$, the determinant equals $1$ iff the set $I$ gives a tree; otherwise it is zero.

Assume now that $M$ is not empty.
$$|Det(q_{ij},  r_{k},  e)| =
\begin{vmatrix} 0 & 0 & \dotsm & -1 & \dotsm & 1\\
                \vdots & \vdots & \ddots & -1 & \dotsm & 1\\
                -1  & 0 & \dotsm & -1 & \dotsm & 1\\
                \vdots & -1 & \dotsm & -1 & \dotsm & 1\\
                1   & 0 & \dotsm & -1 & \dotsm & 1\\
                \vdots & \vdots & \ddots & -1 & \dotsm & 1\\
                0 & 0 & \dotsm & n-1 & \dotsm & 1\\
                \vdots & \vdots & \ddots & -1 & \dotsm & 1\\
                0   & 1 &\dotsm & -1 & \dotsm & 1\\
                \vdots & \vdots & \ddots & -1 & \dotsm & 1\\
                0 & 0 & \dotsm & -1 & \dotsm & 1\\
\end{vmatrix} = $$

Adding  $e$ to all the columns $r_i$, we get:
$$ = n^{|M|} \cdot \begin{vmatrix} 0 & 0 & \dotsm & 0 & \dotsm & 1\\
                \vdots & \vdots & \ddots & 0 & \dotsm & 1\\
                -1  & 0 & \dotsm & 0 & \dotsm & 1\\
                \vdots & -1 & \dotsm & 0 & \dotsm & 1\\
                1   & 0 & \dotsm & 0 & \dotsm & 1\\
                \vdots & \vdots & \ddots & 0 & \dotsm & 1\\
                0 & 0 & \dotsm & 1 & \dotsm & 1\\
                \vdots & \vdots & \ddots & 0 & \dotsm & 1\\
                0   & 1 &\dotsm & 0 & \dotsm & 1\\
                \vdots & \vdots & \ddots & 0 & \dotsm & 1\\
                0 & 0 & \dotsm & 0 & \dotsm & 1\\
\end{vmatrix} =n^{|M|}\cdot(*).$$

We wish to proceed in a similar way, that is, add the columns containing the unique non-zero entry $1$ to other columns chosen in an appropriate way. To explain this reduction let us give two technical definition.

\begin{dfn}\label{DefDecForest}

A \textit{decorated forest} $F=(G,M)$ is a graph $G=([n],I)$ without cycles on $n$ labeled vertices together with a set of marked vertices $M\subset[n]$ such that the following conditions hold:
\begin{enumerate}
    \item number of marked vertices $|M|$  $+$ number of edges  $|I|$ equals $n-1$;
    \item each connected component of $G$ has at most one marked vertex.
\end{enumerate}

\end{dfn}

Immediate observations are:
(1) Each decorated forest has exactly one connected component with no vertices marked. We call it\textit{ a free tree}. Denote by $N(F)$ the number of vertices of the free tree.
(2) Each decorated forest is a disjoint union of the free tree and some rooted forest. The number of rooted trees equals $|M|$.
(3) Each decorated forest $F$ yields a collection of $\{e_{ij}, r_k\}_{(ij)\in I, \ k\in M}$, whose above determinant $(*)$ we denote by $|Det(F)|$ for short.
For instance, for the first decorated forest in Figure \ref{FigKill}, we have  $N(F)=2,\ |M|= 1.$

Now we define the \textit{reduction of a decorated forest} (see Figure \ref{FigKill} for example). It goes as follows.
Assume we have a decorated forest. Take a marked vertex $i$ and an incident edge $(ij)$. Remove the edge and mark the vertex $j$. Repeat until is possible.
Roughly speaking, a marked vertex $i$ {kills} the edge $(ij)$ and {generates} a new marked vertex $j$.

\begin{figure}[h]
\centering
\includegraphics[width=12 cm]{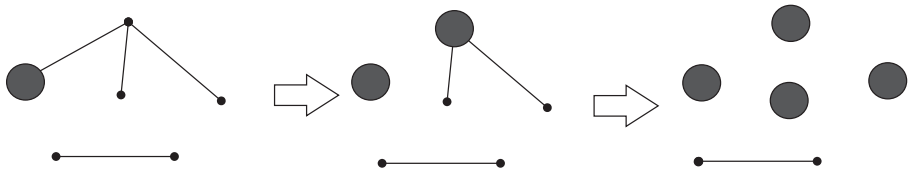}
\caption{Reduction process for a forest with $N(F)~=~2,\ |M|(F)= 1$.
Grey balls denote the marked vertices.}\label{FigKill}
\end{figure}

An obvious observation is:
\begin{lemma}
\begin{enumerate}
    \item The free tree does not change during the reduction.
    \item The reduction brings us to a decorated forest with a unique free tree.  All other trees are one-vertex trees, and all these vertices are marked.
    \item The reduction can be shortened: take the connected components one by one and do the following: if a connected component has a marked vertex, eliminate all its edges and mark all its vertices.
    Otherwise leave the connected component as it is.
    \item The reduction does not depend on the order of the marked vertices we deal with.\qed
\end{enumerate}
\end{lemma}

Before we proceed with the proof of Theorem 1, prove the lemma:
\begin{lemma}\label{LemmaSmallDet}
\begin{enumerate}
    \item For each decorated forest $F$, $$|Det(F)|= N(F).$$
    \item If a collection $\{e_{ij}, r_k\}$ does not come from a decorated forest, that is, violates condition (2) from Definition \ref{DefDecForest}, then
$$|Det(e_{ij}, r_k)|= 0.$$
\end{enumerate}
\end{lemma}

\textit{Proof of the lemma.} \textit{(1)} For a decorated forest, we manipulate with the columns according to the reduction process. We arrive at a matrix which (up to a permutation of the columns and up to a sign) is:
$$\left(
    \begin{array}{ccc}
      A & O &1\\
      O & E &1\\
    \end{array}
  \right).
$$
Here $A$ is the matrix corresponding to the free tree, $E$ is the unit matrix, and the very last column is $e$. Its determinant equals $1$.

\textit{(2)} If the collection of vectors does not yield a decorated forest, that is, there are two marked vertices on one connected component, the analogous reduction gives a zero column.\qed

\bigskip

Basing on Lemmata \ref{LemmaSmallDet} and \ref{lemmaVol}, we  conclude:
$$\Vol( \mathcal{CP}_{n+1}) = \frac{1}{\sqrt{n}}\,\sum_{F } (-n)^{|M(F)|} \cdot N(F), $$
where the sum extends over all decorated forests $F$ on $n$ vertices. (Remind that  $M(F)$ is the set of marked vertices, $N(F)$ is the number of vertices of the free tree.)

Next, we group the forests by the number $N=N(F)$ and write
$$\Vol( \mathcal{CP}_{n+1}) = \frac{1}{\sqrt{n}}\sum_{ N = 1}^{n} \binom{n}{N} N^{N-2} \cdot N \sum_f (-n)^{C(f)} =$$
$$= \frac{1}{\sqrt{n}}\sum_{ N = 1}^{n} \binom{n}{N} N^{N-1}\sum_f (-n)^{C(f)}=\frac{1}{\sqrt{n}}\cdot(**),$$
where the second sum ranges over all rooted forest on $(n-N)$ labeled vertices, $C(\cdot)$ is the number of connected components.

Let us explain this in more details:
(1) $N$ ranges from $1$ to $n$. We choose $N$ vertices in $\binom{n}{N}$ different ways and place a tree on these vertices in $ N^{N-2}$ ways.
    (2) On the rest of the vertices we place a rooted forest.

Since $t_{n-N,k}$ the number of forests on $(n-N)$ labeled vertices of $k$ rooted trees, we write:
$$(**)= \sum_{ N = 1}^{n} \binom{n}{N} N^{N-1}\sum_{k=1}^{n-N} (-n)^{k}\cdot t_{n-N,k}.$$

Section \ref{SecAbel} gives us:
$$\sum_{k = 0}^{n} t_{n,k}x^{k} = A_{n}(x),$$
where $A_{n}(x) = x(x+n)^{n-1}$ is the {Abel polynomial}.

Setting $-n = x$, we get
$$\sum_{k=1}^{n-N} (-n)^{k}\cdot t_{n-N,k} =  A_{n-N}(-n).$$

Thus $(**)$  converts to

$$\sum_{N = 1}^{n}\binom{n}{N} N^{N - 1}A_{n - N}(-n) =: Q_{n}.$$

Applying the definition of $A_{n-N}(-n)$, we get
$$Q_{n}  = \sum_{N = 1}^{n}\binom{n}{N} N^{N - 1} (-n)(- n + n - N)^{n-N-1} =  (-1)^{n} n \cdot \sum_{N = 1}^{n}(-1)^{N}\binom{n}{N} N^{n - 2} .$$

Introduce the following polynomial:
$$p(x) := \sum_{N = 0}^{n}N^{n - 2}\binom{n}{N}  x^{N},$$
for which we have $Q_{n} = p(-1)$. Set also
$$p_{0}(x) := (1 + x)^{n} = \sum_{N = 0}^{n} \binom{n}{N} x^{N},$$
$$p_{i}(x):= x\cdot p'_{i-1}(x) = \sum_{N = 0}^{n} N^i\binom{n}{N}
 x^{N}.$$
We clearly have $p(x) = p_{n-2}(x)$. Besides, $(1+x)^{n - k}$ divides $p_{k}(x)$,  which implies $Q_{n} = 0$. \qed

\section{Homologies of the face poset of cyclopermutohedron}\label{SectHomol}

Since the face poset of $\mathcal{CP}_{n+1}$ is isomorphic to the complex $CP_{n+1}$ (defined in the Introduction),
 in the section we shall work with the latter complex.

 Let us make the following  conventions that are illustrated in Figure \ref{FigExLabel}.
First remind that $k$-cells  of the complex are labeled by (all possible) cyclically ordered partitions of the set $[n+1]=\{1,...,n+1\}$ into $(n-k+1)$ non-empty parts,
so the number of parts is at least $3$.

\begin{enumerate}
  \item Instead of ''the cell of the complex $CP_{n+1}$ labeled by $\alpha$'' we say for short ''the cell $\alpha$''.

  \item For a cell $\alpha$ the \textit{$(n+1)$-set} is the set in the partition containing the entry $n+1$.
  \item We represent a cyclically ordered partition $\alpha$ as a linearly ordered partition of the same set $[n+1]$ by cutting the circle right after the  $(n+1)$-set. For example, the labels depicted in   Fig. \ref{FigExLabel} we write as $(\{6\} \{1,3,5\}\{2\} \{4,7\})$ and $(\{6\} \{1,2,3,5\} \{4,7\})$ .

In particular, the vertices of the complex ${CP}_{n+1}$ are labeled by (all possible) permutations of the set $[n+1]$ ending with the entry $n+1$.
Therefore we have a map $\sigma: \V(CP_{n+1}) \rightarrow S_{n}$: the removal of $\{n+1\}$ from the label gives an element of the symmetric group  $S_{n}$.
  \item For a cell $\alpha$ we denote by $|\alpha|$ the number of parts in the partition. Remind that we always have  $|\alpha|\geq 3$.
  \item For $i\in [n+1]$ and a set $I \subset [n+1]$ we write $i<I$ whenever $j\in I$ implies $i<j$.
    \item A \textit{singleton} is a one-element set.
\end{enumerate}

\subsection{Discrete Morse function on the complex $CP_{n+1}$}
Let us introduce a discrete Morse function  on the cells of the complex $CP_{n+1}$.
It is going to be a \textit{perfect} Morse function, and therefore will give us the homology groups directly.

{\textbf{Step 1.}}
We pair together two cells
$$\alpha=\big(\cdots \ \{1\}\ I\ \cdots\big)  \hbox{ and } \beta=\big(\cdots\ \{1\}\cup I\ \cdots\big)$$ iff  $n+1\notin I$.

\medskip

We proceed for all $2 \leq k< n$, assuming that the $k$-th step is:

{\textbf{Step k.}}
We pair together two cells
$$\alpha=\big(\cdots\ \{k\}\ I\ \cdots\big)  \hbox{ and } \beta=\big(\cdots\ \{k\}\cup I\ \cdots\big)$$
iff the following holds:
\begin{enumerate}
	\item $\alpha$ and $\beta$ were not paired at any of the previous steps.
    \item  $n+1\notin I$.
    \item $k<I$.
\end{enumerate}

\medskip

\textbf{Examples.}
The cell $(\{2\} \ \{4,3\} \ \{1\} \ \{5,6\})$ is paired with the cell  $(\{2,4,3\} \ \{1\} \ \{5,6\})$ on the second step.
The cell $(\{4\} \ \{5\} \ \{3\} \ \{1\} \ \{6,2\})$ is paired with the cell  $(\{4, 5\} \ \{3\} \ \{1\} \ \{6,2\})$ on the fourth step.
The cell $(\{4\} \ \{3\} \ \{2\} \ \{1\} \ \{5,6\})$ is not paired.

\medskip

It is convenient to reformulate the above  as an algorithm for finding a pair  for a given cell $\alpha$ (if such a pair exists).

\medskip

\textbf{Pair search algorithm.}

Take a cell $\alpha$.
\begin{enumerate}
\item[I.] We call an entry  $k \in [n]$ \textit{forward-movable in} $\alpha$ if
	\begin{enumerate}
  		\item[$1^{\circ}$] $\alpha$ consists of more than three subsets;
  		\item[$2^{\circ}$] $k$ forms a singleton in $\alpha$;
  		\item[$3^{\circ}$] the singleton $\{k\}$ is followed by a set $I$ such 	that  $k<I$ and $n+1\notin I$.
	\end{enumerate}
\item[II.] We call an entry $k \in [n]$ \textit{backward-movable in} $\alpha$ if
	\begin{enumerate}
		\item[$1^{\diamondsuit}$] the entry $k$ lies in $\alpha$ in a non-singleton set $J$, such that $n+1\notin J$;
		\item[$2^{\diamondsuit}$] $k = min(J)$;
		\item[$3^{\diamondsuit}$] one of the following conditions holds:
 			\begin{enumerate}
				\item the set $J$ is preceded by a non-singleton set;
				\item the set $J$ is preceded by a singleton $\{m\}$ with $m > k$;
				\item the set $J$ is preceded by a set containing $n+1$.
			\end{enumerate}
	\end{enumerate}
\end{enumerate}

In this notation, the algorithm looks as follows:
\begin{enumerate}
  \item If a cell $\alpha$ has no movable entries, $\alpha$ is not paired.
  \item Assuming that a cell $\alpha$ has  movable entries, take the minimal movable (either forward or backward) entry $k$ in $\alpha$. Then the cell $\alpha$ is paired with a cell that is formed from $\alpha$ by moving $k$ either forward inside the next set, or backward out of the set containing $k$, according to the $k$-th step of pairing algorithm.
\end{enumerate}

\begin{lemma}
The above pairing  is a discrete Morse function.
\end{lemma}
\begin{proof} The conditions  ''each cell of the complex participates in at most one pair'', and
   ''in each pair, the cell $\alpha^p$ is a facet of the cell $\beta^{p+1}$'' hold automatically. The  ''no closed paths'' condition follows from the observation that no two entries  interchange their order during a path more than once.
\end{proof}

\begin{lemma}\label{Lemma_CritCells}
The critical cells of the above defined Morse function are exactly all the cells of the following two types:

\textbf{Type 1.} Cells labeled by
$(\; \spadesuit \  \{n+1, \dots\}\ ) ,$
where $\spadesuit$ is a string of singletons coming in decreasing order.

\textbf{Type 2.} Cells labeled by
$(\; \{i\}\   I \   \{n+1, \dots\}\;),$
where  $i<I$.

\end{lemma}

\begin{proof} These are exactly all the labels that do not have movable entries.
\end{proof}

\textbf{Examples.}
$(\{4\} \ \{3\} \ \{2\} \ \{1\} \ \{5,6\}) $ is a critical cell of type 1,
$(\{1\} \ \{2,4,3\} \ \{5,6\})$ is a critical cell of type 2,
$(\{1\} \ \{2\} \ \{3,4,5,6\})$ is a critical cell of type 2.

\medskip

\begin{lemma}\label{LemmaNumCells}
For the discrete Morse function  there are exactly
$$2^n+\frac{n^2-3n-2}{2}$$
critical cells of dimension $n-2$, and  exactly
$$\binom{n}{k}$$
critical cells of dimension $k$ for all $0\leq k < n-2$.\qed
\end{lemma}

\begin{lemma}\label{LemmaDIfferentials}
The boundary operators of the Morse complex  vanish.
\end{lemma}
Theorem 2  \ follows from the above two lemmata and  Section \ref{SSecDMF}.

\bigskip

The proof of the Lemma \ref{LemmaDIfferentials} is contained in the next  section.
The detailed proof  is somewhat technical, but the idea is simple: we show that for each pair of critical cells $\alpha$ and $\beta$ either there are no gradient paths
leading from $\beta$ to $\alpha$, or there are exactly two paths   coming with different orientations.

\section{Proof of Lemma \ref{LemmaDIfferentials}}\label{SecProofs}

\subsection{Gradient paths}
Let us describe the gradient paths connecting critical cells $\beta^{p+1}$ and $\alpha^{p}$.

\begin{lemma}
There are no critical gradient paths  that end at critical cells of type $2$.
\end{lemma}
Proof. Critical cells of type $2$ have the maximal possible dimension.
\qed

\begin{lemma}\label{lemmaCriPaths}
The three following cases describe all gradient paths joining critical cells:

\begin{enumerate}
\item   $\hbox{ Let }\beta=( \spadesuit \    \{n+1, \dots\} ) \hbox{ and } \alpha=( \heartsuit \  \{n+1, \dots\} )$ be two cells of type 1. Then
	there are two gradient paths from $\beta$ to $\alpha$ iff $\heartsuit = \spadesuit \cup {k}$ for some $k$.
\item Let $\beta=(\{i\} \  \{j,k\} \    \{n+1, \dots\}) \hbox{ and } $ \newline $\alpha=(\{k\} \ \{j\} \ \{i\} \ \{n+1, \dots\} )$ be cells of type 2 and 1 respectively. Then there are two gradient paths from $\beta$ to $\alpha$.
    \item Let $\beta=( \{i\} \  \{j\} \  \{n+1, \dots\}) \hbox{ and } \alpha=( \spadesuit \  \{n+1, \dots\})$ be cell of type 2 and 1 respectively. Then there are two gradient paths from $\beta$ to $\alpha$ iff  $\spadesuit$ consists of three singletons, two of which are $\{i\}$ and $\{j\}$.	
	
\end{enumerate}
\end{lemma}

We start with  examples:

\begin{enumerate}

  \item  $\hbox{ For }\beta=(\{5\} \  \{3\} \  \{1\} \ \{6,4,2\}) \hbox{ and } \alpha=(\{5\} \ \{3\} \ \{2\} \ \{1\} \ \{6,4\} )$
the two paths are:

$$ (\{5\} \ \{3\} \ \{1\} \  \{6,4,2\} )$$
$$ (\{5\} \  \{3\} \  {\{1\}} \  \{2\} \  \{6,4\} )$$	
$$ (\{5\} \  \{3\} \  \{1,2\} \  \{6,4\} )$$
$$ (\{5\} \  \{3\} \  \{2\} \  \{1\} \  \{6,4\} )$$
and
$$(\{5\} \ \{3\} \ \{1\} \ \{6,4,2\} )$$
$$({\{2\}} \  \{5\} \  \{3\} \  \{1\} \ \{6,4\} )$$
$$(\{2,5\} \  \{3\} \ \{1\} \  \{6,4\} )$$
$$(\{5\} \  {\{2\}} \  \{3\} \  \{1\} \ \{6,4\})$$
$$(\{5\} \  \{2,3\} \  \{1\} \ \{6,4\} )$$
$$(\{5\} \  \{3\}  \ \{2\}  \ \{1\} \ \{6,4\} ).$$

\medskip

  \item
$\hbox{ For } \beta =(\{i\} \ \{j,k\} \  \{n+1, \dots\}) \hbox{ and } \alpha = (\{k\} \ \{j\} \ \{i\}  \  \{n+1, \dots\}), $ where $i<j<k$, the two  paths are:
$$ ( \{i\} \  \{j, k\} \  \{n+1, \dots\})$$
$$ ( {\{i\}} \  \{j\}  \ \{k\} \ \{n+1, \dots\})$$
$$ ( \{i,j\} \  \{ k\} \  \{n+1, \dots\})$$
$$ ( \{j\}  \ {\{i\}}  \ \{k\} \ \{n+1, \dots\})$$
$$ ( \{j\} \  \{i,k\} \ \{n+1, \dots\})$$
$$ ( {\{j\}} \  \{k\} \  \{i\} \ \{n+1, \dots\})$$
$$ ( \{j,k\} \  \{i\} \ \{n+1, \dots\})$$
$$ ( \{k\}  \ \{j\}  \ \{i\} \ \{n+1, \dots\})$$
and
$$ ( \{i\} \   \{j, k\}  \ \{n+1, \dots\})$$
$$ ( {\{i\}}  \ \{k\} \    \{j\} \  \{n+1, \dots\})$$
$$ (\{i,k\} \    \{j\}  \ \{n+1, \dots\})$$
$$ ( \{k\} \ {\{i\}}  \   \{j\} \  \{n+1, \dots \})$$
$$ (  \{k\}  \   \{i,j\} \  \{n+1, \dots\})$$
$$ ( \{k\} \  \{j\} \  \{i\} \ \{n+1, \dots\}).$$

\medskip

  \item
  $\hbox{For }\beta =( \{i\}  \ \{j\}  \ \{n+1, \dots\, k\}) \hbox{ and } \alpha = (  \spadesuit \  \{n+1, \dots\}) $ there are three possible cases:

\textbf{Case 1.}
For  $k<i<j$ the two paths from $\beta$ to $\alpha$ are:
$$ ( \{i\}  \ \{j\}  \ \{n+1, \dots\, k\})$$
$$ ( {\{i\} } \  \{j\} \  \{k\} \ \{n+1, \dots\})$$
$$ ( \{i,j\} \  \{k\} \ \{n+1, \dots\})$$
$$ (  \{j\} \  \{i\}  \ \{k\} \ \{n+1+1, \dots\})$$
and
$$ ( \{i\} \  \{j\} \  \{n+1, \dots\, k\})$$
$$ ({\{k\} } \ \{i\} \  \{j\}  \  \{n+1, \dots\})$$
$$ ( \{k,i\} \  \{j\} \  \{n+1, \dots\})$$
$$ ( \{i\} \ {\{k\} } \   \{j\} \  \{n+1, \dots\})$$
$$ ( \{i\} \  \{k,j\}  \ \{n+1, \dots\})$$
$$ (  \{j\} \  \{i\} \  \{k\} \ \{n+1, \dots\}).$$

\medskip
	
\textbf{Case 2.}  For  $i<k<j$ the two paths are:
$$ ( \{i\} \  \{j\} \  \{n+1, \dots\, k\})$$
$$( {\{i\} } \ \{j\}  \ \{k\} \ \{n+1, \dots\})$$
$$( \{i,j\} \  \{k\} \ \{n+1, \dots\})$$
$$(  \{j\} \  {\{i\}} \ \{k\} \ \{n+1, \dots\})$$
$$( \{j\} \  \{i,k\} \ \{n+1, \dots\})$$
$$( \{j\}  \ \{k\} \ \{i\} \ \{n+1, \dots\})$$
and
$$( \{i\} \  \{j\} \  \{n+1, \dots\, k\})$$
$$(\{k\} \ {\{i\} } \  \{j\} \   \{n+1, \dots\})$$
$$( \{k\} \  \{i,j\} \  \{n+1, \dots\})$$
$$({\{k\}} \ \{j\} \{i\} \  \{n+1, \dots\})$$
$$( \{k,j\} \  \{i\} \  \{n+1, \dots\})$$
$$( \{j\} \  \{k\} \ \{i\} \ \{n+1, \dots\}).$$

\medskip

\textbf{Case 3.} For $i<j<k$ the two paths are:
$$ ( \{i\} \  \{j\}  \ \{n+1, \dots\, k\})$$
$$( {\{i\} } \ \{j\}  \ \{k\} \ \{n+1, \dots\})$$
$$( \{i,j\} \  \{k\} \ \{n+1, \dots\})$$
$$(  \{j\} \ {\{i\}} \  \{k\} \ \{n+1, \dots\})$$
$$( \{j\} \  \{i,k\} \ \{n+1, \dots\})$$
$$({ \{j\} } \ \{k\} \ \{i\} \ \{n+1, \dots\})$$
$$( \{j,k\} \ \{i\} \ \{n+1, \dots\})$$
$$( \{k\} \  \{j\} \ \{i\} \ \{n+1, \dots\})$$
and
$$( \{i\} \  \{j\} \  \{n+1, \dots\, k\})$$
$$(\{k\} \ {\{i\} } \  \{j\}  \  \{n+1, \dots\})$$
$$( \{k\}  \ \{i,j\}  \ \{n+1, \dots\})$$
$$( \{k\} \ \{j\} \ \{i\}  \  \{n+1, \dots\}).$$
\end{enumerate}

\bigskip

Now prove the lemma.
\begin{proof}
Consider the case 1.
Suppose there is a path from $\beta$ to $\alpha$.
	Then $dim(\beta)= dim(\alpha) +1$, therefore, $|\heartsuit| = |\spadesuit| +1$.
	Since no entry joins the $(n+1)$-set during the path, we have $
	\heartsuit = \spadesuit \cup {k}$ for some k.
	
	Now we have $$\beta=(\spadesuit \  \{n+1, \dots\, k\} )$$ and $$\alpha=( \spadesuit \cup {k} \  \{n+1, \dots\}).$$
The existent gradient paths come from a simple case analysis.
The other two cases are proved in the similar way.
    \end{proof}

\subsection{Canonical orientations of cells}\label{SecOr}

Remind that two vertices of $CP_{n+1}$ are joined by an edge whenever their labels differ on a permutation of two neighbor entries. Such vertices we will call \textit{neighbors}.
For a cell $\alpha = (I_{1}  I_{2} \dots I_{l})$, a vertex  $V \in \alpha$  has  exactly   $\dim(\alpha)$  neighbors  that belong to $\alpha$. The latter are called  $\alpha$-\textit{neighbors} of $V$.
For every  $V \in \alpha$ we order the  $\alpha$-neighbors of $V$: we get the first neighbor of $V$ by interchanging the first\footnote{We use ''from left to right'' orientation on linearly ordered labels.} two entries of $V$ that belong to the same set $I_i$, the second neighbor we get by interchanging the second two entries of $V$ that belong to the same $I_i$, etc.
This ordering defines the \textit{orientation of the cell $\alpha$ related to the vertex }$V$.  Here we explore the following observation: the cells of the complex are combinatorially isomorphic to the product of permutohedra, and therefore can be realized as some convex polytopes.
More precisely, a cell labeled by $(I_1,...,I_m)$ is combinatorially isomorphic to $\Pi_{|I_1|}\times... \times \Pi_{|I_m|}$.
To fix an orientation on a polytope, it suffices to fix an order on all the vertices that are neighbors of a fixed vertex.

\textit{The principal vertex} $PR(\alpha)$ of the cell $\alpha$ is the vertex with the label $(\widehat{I_{1}} , \widehat{I_{2}}, \dots , \widehat{I_{l}})$, where $\widehat{I_{j}}$ is a partition of the set $I_{j}$ into singletons coming in increasing order.
The orientation of the cell $\alpha$  related to its principal  vertex  $PR(\alpha)$  is called the  \textit{the canonical orientation} of the cell $\alpha$.

\medskip

\textbf{Examples.}

(1) For the cell $\alpha = (\{1\} \ \{2, 4, 5\} \ \{3\} \ \{6, 7, 8\})$  and  its vertex \newline $V=(\{1\} \ \{4\} \ \{5\} \ \{2\} \ \{3\} \ \{7\} \ \{6\} \ \{8\})$
the $\alpha$-neighbors of $V$ are ordered as follows:

$V_{1} = (\{1\} \ \{5\} \ \{4\} \ \{2\} \ \{3\} \ \{7\} \ \{6\} \ \{8\} ),$

$V_{2} = (\{1\} \ \{4\} \ \{2\} \ \{5\} \ \{3\} \ \{7\} \ \{6\} \ \{8\} ),$

$V_{3} = (\{1\} \ \{4\} \ \{5\} \ \{2\} \ \{3\} \ \{6\} \ \{7\} \ \{8\} ), \  \hbox{etc.}$

\medskip

(2) For the cell $\alpha = (\{1,4,5\} \  \{2,3,7\} \  \{6,8\})$ the principal vertex is
$PR(\alpha) = (\{1\} \ \{4\} \ \{5\} \ \{2\} \ \{3\} \ \{7\} \ \{6\} \ \{8\}).$

\medskip

For a cell $\alpha$ and its vertex $V \in \V(\alpha)$, the permutation $\sigma_{V,\alpha} \in S_{n}$ is defined by
$$\sigma_{V,\alpha} \circ{\sigma(PR(\alpha))} =\sigma(V).$$

\begin{lemma}\label{lemmaChangeOr}
In the above notation the orientation of  a cell $\alpha$ related to a vertex $V$ equals $\sign(\sigma_{V,\,\alpha})$.\qed
\end{lemma}

\subsection{Boundary operators vanish }
\label{sec:CA}

Now we are ready to prove Lemma \ref{LemmaDIfferentials}.
As we have seen, each pair of critical cells is connected either by no paths or by exactly two paths. To show that in the latter case the paths come with different orientations (this is exactly what the lemma states) we analyze elementary steps in two auxiliary lemmata.

Assume we have a gradient  path
$$\beta^{p+1}_0,\, \alpha^p_1,\,\beta^{p+1}_1,\,\alpha^p_2,\,\beta^{p+1}_2,\dots,\,\alpha^p_k,\,\beta^{p+1}_k,\,\alpha^p_{k+1},$$
with  $\beta^{p+1}_0$ and $\alpha^p_{k+1}$  critical. By definition, two consecutive $\beta^{p+1}_i$ and  $\beta^{p+1}_{i+1}$ share a facet $\alpha^p_i$. To compute the sign of this path, we  compare the canonical orientations of $\beta^{p+1}_i$ and  $\beta^{p+1}_{i+1}$ . We also need to compare   the orientations of the cells $\beta^{p+1}_k$ and $\alpha^p_{k+1}$. We  are especially interested in the first steps of the paths (which can be of both types).

\bigskip

For a cell $ \beta$ and $k \in [n]$, denote by $N(\beta,\  k)$ (respectively, $M(\beta,\  k)$) the number of entries in the $(n+1)$-set of $\beta$ which are bigger (respectively, smaller) than $k$.

\begin{lemma} \label{lemmaFirstStep}
(First steps)
Now suppose we have two critical cells $\beta$ and $\alpha$ connected by two paths\footnote{As is described in  Lemma \ref{lemmaCriPaths}}. 
Then exactly one of the first steps in these paths has  disagreement in the canonical orientations.

More precisely, we have the following.

   \begin{enumerate}

   \item
   \begin {enumerate}

	\item    For $k>i$ the canonical orientations of the cells  $$\beta = (\spadesuit \  \{i\} \  \{n+1,k \dots\}) \hbox{ and } \beta'=(\spadesuit \  \{i,k\} \  \{n+1 \dots\} )$$
        agree iff $N(\beta,\  k)$ is odd.
	\item For $k<i$ the canonical orientations of the cells $$\beta = (\{i\} \ \spadesuit \  \{n+1,k \dots\}) \hbox{ and } \beta'=(\{i,k\} \ \spadesuit \ \{n+1 \dots\} )$$
     agree iff $M(\beta,\  k)$  and      $dim(\beta)$ have  different parity.

     	\item For $k>i$ the canonical orientations of the cells  $$\beta = (\spadesuit \    \{n+1,k \dots\}) \hbox{ and } \alpha=(\spadesuit \  \{k\} \  \{n+1 \dots\} ).$$
                 agree iff  $N(\beta,\  k)$  is odd.
        	\item For $k<i$ the canonical orientations of the cells $$\beta = ( \spadesuit \  \{n+1,k \dots\}) \hbox{ and } \alpha=(\{k\} \ \spadesuit \ \{n+1 \dots\} )$$
            agree iff  $M(\beta,\  k)$ and $dim(\beta)$ have  different parity.
\end{enumerate}
\item
 \begin {enumerate}
    \item  For $i<j<k$ the canonical orientations of the cells $$\beta=(\{i\} \ \{ j,k\} \ \dots) \hbox{ and } \beta'= (\{i,j\} \ \{k\} \ \dots)$$ always disagree.

     \item For $i<j<k$   the canonical orientations of the cells $$\beta = (\{i \} \ \{j,k\}\ \dots)\hbox{ and } \beta'=(\{i, k\} \ \{j\} \ \dots )$$ always agree.
\end{enumerate}
\item  \begin {enumerate}
    \item  For $i<j$ the canonical orientations of the cells $$\beta=(\{i\} \ \{ j\} \ \{n+1,k \dots\}) \hbox{ and } \beta'= (\{i,j\} \ \{k\} \ \{n+1 \dots\})$$ agree iff $N(\beta,\  k)$ is odd.

     \item  For $i<j$ the canonical orientations of the cells $$\beta=(\{i\} \ \{ j\} \ \{n+1,k \dots\}) \hbox{ and } \beta'= ( \{k\} \ \{i,j\} \ \{n+1 \dots\})$$ agree iff $M(\beta,\  k)$  and      $dim(\beta)$ have  different parity.
\end{enumerate}

       \end{enumerate}
       \end {lemma}

\begin{proof}
We give the proof of some cases.

(1.a)
Note that $\Cen(\beta') \in \beta$. We have
$$\Cen(\beta) = (\dots \  \{i\}\dots \  \{k\} \ \dots \ \{n+1\dots\}),\hbox{ and}$$
 $$\Cen(\beta') = (\dots \{i\} \ \{k\} \ \dots \  \{n+1\dots\}).$$

There are exactly $N(\beta,\  k)$ elementary transpositions that turn   $\Cen(\beta)$ to $\Cen(\beta')$. So, by Lemma \ref{lemmaChangeOr} the orientation  associated with $\Cen(\beta')$ in $\beta$ is  positive iff $N(\beta,\  k)$ is even.
Observe also that orientation  at vertex $\Cen(\beta')$ for the cells $\beta$ and $\beta'$ are opposite.

(1.b)
Take the vertex $\Cen(\beta')= ( \{k\}\{i\} \ \dots \  \{n+1\dots \})$. It belongs to the cell $\beta$.  $\Cen(\beta')$ differs from $PR(\beta)$ by $M(\beta,\  k)$ elementary transpositions. Therefore, by Lemma \ref{lemmaChangeOr},  the orientation, associated with $PR(\beta')$ in $\beta$ is positive iff $M(\beta,\  k)$ is even.

  Now consider the orientation of $\beta'$. If we denote the $\beta$-neighbors of $\Cen(\beta')$  by $$ A_1,A_2,\dots , A_{\dim(\beta)},$$ then its $\beta'$-neighbors are $$A', A_1, A_2 \dots, A_{\dim(\beta)-1} ,$$ where $A'\in \beta'$. It is easy to see that the orientation, associated with $A$ in $\beta'$ and $\beta$ agree iff $\dim(\beta)$ is even.

(2.a) $\Cen(\beta')= (\spadesuit \{i\}\ \{k\} \ \{j\} \ \dots ) \in \beta$.  $\Cen(\beta')$ differs from $PR(\beta)$ by one elementary transposition. If we denote the $\beta$-neighbors of $\Cen(\beta')$  by $$ A_1,A_2,\dots , A_{\dim(\beta)},$$ then its $\beta'$-neighbors are $$A', A_2 \dots, A_{\dim(\beta)-1},$$ where $A'\in \beta'$.

All other cases are treated analogously.
\end{proof}

 The following lemma is proved in the similar way as Lemma \ref{lemmaFirstStep}.

\begin{lemma} \label{lemmaElemStep}
(Intermediate and last steps) Assume that two critical cells $\beta$ and $\alpha$ connected by two paths.
At  all the steps of gradient (except for the first steps) paths canonical orientations always agree.
	
\end{lemma}

\bigskip

Now  the
proof of the Lemma \ref{LemmaDIfferentials} comes from  the two above lemmata.

\end{document}